\documentclass{article}

\usepackage{amsmath}
\usepackage{amsthm}
\usepackage{amssymb}

\usepackage[text={6.5in,8in},centering]{geometry}

\usepackage[colorlinks=true,urlcolor=blue,linkcolor=blue,plainpages=false,pdfpagelabels]{hyperref}

\newtheorem{theorem}{Theorem}[section]
\newtheorem{definition}[theorem]{Definition}
\newtheorem{proposition}[theorem]{Proposition}
\newtheorem{corollary}[theorem]{Corollary}
\newtheorem{lemma}[theorem]{Lemma}
\newtheorem{example}[theorem]{Example}

\newcommand{\bdfn}{\begin{definition}}
\newcommand{\edfn}{\end{definition}}
\newcommand{\bthm}{\begin{theorem}}
\newcommand{\ethm}{\end{theorem}}
\newcommand{\bprop}{\begin{proposition}}
\newcommand{\eprop}{\end{proposition}}
\newcommand{\bcor}{\begin{corollary}}
\newcommand{\ecor}{\end{corollary}}
\newcommand{\blem}{\begin{lemma}}
\newcommand{\elem}{\end{lemma}}
\newcommand{\bex}{\begin{example}\begin{rm}}
\newcommand{\eex}{\end{rm}\end{example}}

\newcommand{\be}{\begin{enumerate}}
\newcommand{\ee}{\end{enumerate}}
\newcommand{\bce}{\begin{center}}
\newcommand{\ece}{\end{center}}
\newcommand{\ba}{\begin{array}} 
\newcommand{\ea}{\end{array}}
\newcommand{\beq}{\begin{equation}}
\newcommand{\eeq}{\end{equation}}
\newcommand {\bua} {\begin{eqnarray*}}
\newcommand {\eua} {\end {eqnarray*}}
\newcommand {\bea} {\begin{eqnarray}}
\newcommand {\eea} {\end {eqnarray}}
\newcommand{\bt}{\begin{tabular}}
\newcommand{\et}{\end{tabular}}
\newcommand{\bi}{\begin{itemize}}
\newcommand{\ei}{\end{itemize}}

\newcommand{\Ra}{\Rightarrow}
\newcommand{\lan}{\left\langle}
\newcommand{\ran}{\right\rangle}
\newcommand{\eps}{\varepsilon}
\newcommand{\vp}{\varphi}
\newcommand{\ds}{\displaystyle}
\newcommand{\N}{{\mathbb N}}
\newcommand{\se}{\subseteq}
\newcommand{\limn}{\ds\lim_{n\to\infty}}
\newcommand{\lsupn}{\ds \limsup_{n\to\infty}}

\newcommand{\rateczeroq}{\sigma_{0}}
\newcommand{\rateczeroqs}{\sigma^*_{0}}
\newcommand{\rateconeq}{\sigma_{1}}
\newcommand{\ratectwoq}{\sigma_{2}}
\newcommand{\ratecthreeq}{\sigma_{3}}
\newcommand{\ratecfourq}{\sigma_{4}}
\newcommand{\ratecfiveq}{\sigma_{5}}
\newcommand{\ratecsixq}{\sigma_{6}}
\newcommand{\lsn}{\delta}

\newcommand{\exasreg}{\overline{\Sigma}}
\newcommand{\exLmeta}{\overline{\Delta}_L}

\title{Quantitative results on a Halpern-type proximal point algorithm\thanks{This is a preprint of an article published in Computational Optimization and Applications. The final authenticated version is
available online at: \url{https://doi.org/10.1007/s10589-021-00263-w}.}
}
\author{Lauren\c{t}iu Leu\c{s}tean${}^{a,b}$  and Pedro Pinto${}^{c}$\\[2mm]
\footnotesize ${}^a$ Research Center for Logic, Optimization and Security (LOS), Department of Computer Science, \\
\footnotesize Faculty of Mathematics and Computer Science, University of Bucharest.\\
\footnotesize  Academiei 14,  010014 Bucharest, Romania\\[1mm]
\footnotesize ${}^b$ Simion Stoilow Institute of Mathematics of the Romanian Academy,\\
\footnotesize Calea Grivi\c tei 21, 010702 Bucharest, Romania \\[1mm]
\footnotesize ${}^c$ Department of Mathematics, Technische Universit\" at Darmstadt,\\
\footnotesize Schlossgartenstra\ss{}e 7, 64289 Darmstadt, Germany\\[2mm]
\footnotesize E-mails:  laurentiu.leustean@unibuc.ro, pinto@mathematik.tu-darmstadt.de
}

\date{}
\begin{document}

\maketitle

\begin{abstract}
We apply proof mining methods to analyse a  result of Boikanyo and Moro\c{s}anu on the strong convergence of a Halpern-type proximal point algorithm. As a consequence, we obtain quantitative versions of this result, providing uniform effective rates of asymptotic regularity and metastability. \\



\end{abstract}

\section{Introduction}

Let $H$ be a real Hilbert space  and $A:H\to 2^{H}$ be a maximally monotone operator such 
that the set $zer(A)$ of zeros of $A$ is nonempty. For every $\gamma>0$, the resolvent $J_{\gamma A}$ of 
$\gamma A$  is defined by $J_{\gamma A}= (id_H + \gamma A)^{-1}$. It is well-known (see, e.g., \cite{BauCom17}) 
that $J_{\gamma A}:H\to H$ is a single-valued firmly nonexpansive (hence, nonexpansive) mapping and that 
$Fix(J_{\gamma A})=zer(A)$  for every $\gamma>0$. Furthermore, $zer(A)$ is a closed convex subset of $H$ and
$P_{zer(A)}$ denotes the projection onto $zer(A)$.

A major problem in convex optimization is finding zeros of maximally monotone operators. A classical method
for solving this problem is the  proximal point algorithm, defined by 
Rockafellar \cite{Roc76} as follows: 
\begin{equation}\label{PPA}
PPA \qquad x_0\in H, \quad x_{n+1}:=J_{\beta_n A}x_n+e_n,
\end{equation}
where $(\beta_n)_{n\in\N}$ is  a sequence of positive real numbers and $(e_n)_{n\in\N}\se H$ 
is a sequence of errors.  Special cases of \eqref{PPA} 
have been previously studied by Martinet \cite{Mar70}. 
Rockafellar proved, under the assumptions that  $(\beta_n)$ is bounded away from zero and $(\|e_n\|)$ is a summable sequence, that 
$(x_n)$ is weakly convergent to a zero of $A$ and he posed the question whether the weak convergence 
can be improved, in general, to strong convergence. This question was answered in the negative by G{\" u}ler \cite{Gul91}.

This being the case, the following problem is very natural:  
\bce 
modify PPA such that strong convergence is guaranteed.
\ece
This problem has attracted a lot of research, many new algorithms based on PPA were introduced and 
proved to be strongly convergent (just to give a few examples, see \cite{EckBer92,SolSva00,Xu02,WanWanXu16,BotCseMei19}). 
Since the set of zeros of  a maximally monotone operator coincides with the fixed point set of the resolvent, 
one idea to obtain new algorithms is to combine PPA with nonlinear iterations studied in metric fixed point theory. 
One such iteration is the well-known Halpern iteration, defined, for any nonexpansive mapping $T:H\to H$ and any sequence $(\alpha_n)_{n\in\N}$ in $[0,1]$, by 
\beq\label{def-Halpern}
x_0,u\in H, \quad x_{n+1}:=\alpha_nu+(1-\alpha_n)Tx_n.
\eeq
The iteration was introduced by Halpern \cite{Hal67} for the special case $u=0$. A classical 
result is Wittmann's theorem \cite{Wit92}, which proves the strong convergence of $(x_n)$ towards a fixed point of $T$, under some assumptions on  $(\alpha_n)$ that are satisfied for the natural choice 
$\alpha_n=\frac1{n+1}$.  Since, for $T$ linear  and $\alpha_n=\frac1{n+1}$, the Halpern 
iteration  becomes the ergodic average, Wittmann's result  is a nonlinear generalization of 
the von Neumann mean ergodic theorem.

By combining PPA with the Halpern iteration we obtain the so-called {\em Halpern-type proximal point algorithms}. 
One such algorithm was introduced independently by Kamimura and 
Takahashi \cite{KamTak00} and Xu \cite{Xu02}:   
\begin{equation}\label{HPPA-def-1}
HPPA \qquad  x_0,u\in H\, \quad x_{n+1}:=\alpha_n u+(1-\alpha_n)J_{\beta_n A}x_n+e_n,
\end{equation}
where $(\alpha_n)$ is a sequence in $(0,1]$, $(\beta_n)$ is a sequence of positive real numbers 
and $(e_n)\se H$ is the error sequence. Strong convergence results for HPPA  and its generalizations to classes of Banach or geodesic spaces were obtained by imposing different conditions on the sequences $(\alpha_n)$, $(\beta_n)$, ($e_n)$.

We consider in the sequel these  conditions: 
$$\ba{llll}
\!\!\!\!\!\!\!(C0) \,  \limn \alpha_n=0, &  \,\, (C1)  \, \sum\limits_{n=0}^\infty \alpha_n=\infty, 
& \,\, (C2) \,  \prod\limits_{n=0}^\infty (1-\alpha_n)=0, \\
\!\!\!\!\!\!(C3) \,  \limn \frac{|\alpha_{n+1}-\alpha_n|}{\alpha_n^2} =0, &
\,\, (C4)\,  \limn\beta_n=\beta >0, &  
\,\, (C5) \, \sum\limits\limits_{n=0}^\infty \|e_n\| <\infty,  \\
\!\!\!\!\!\!(C6) \,  \limn\frac{\|e_n\|}{\alpha_n}=0.
\ea
$$

The following strong convergence result was proved by Boikanyo and Moro-{\c s}anu \cite{BoiMor11}.

\bthm\label{main-theorem}
Let $H$ be a Hilbert space, $A:H\to 2^H$ be a maximally monotone operator such that $zer(A)\ne\emptyset$ and 
$(x_n)$ be the HPPA defined by \eqref{HPPA-def-1}. Assume that the following hold:
\be
\item $(C0)$, $(C3)$ and $(C4)$;
\item $(C1)$ or, equivalently, $(C2)$; 
\item $(C5)$ or $(C6)$.
\ee
Then  $(x_n)$  converges  strongly to  $P_{zer(A)}u$.
\ethm

As one is interested in applications of HPPA  to practical problems, an important direction of research is the study of their complexity, 
hence obtaining explicit bounds on their convergence, on their asymptotic behaviour. The main results of this paper are such effective bounds. 

Firstly, we compute rates of asymptotic regularity of $(x_n)$, which turn out to be polynomial for the example we give in Section~\ref{hppa-example}. 
Furthermore, we prove quantitative versions of Theorem~\ref{main-theorem}, providing uniform rates of metastability (in the sense of Tao \cite{Tao07,Tao08}). 
As pointed out by Kohlenbach \cite{Koh20a},  Neumann \cite{Neu15} proves that one cannot obtain computable rates of convergence even for the simple case $H={\mathbb R}$.  Results from mathematical logic show that one can 
extract effective rates of metastability of $(x_n)$, defined as mappings $\Phi:\N \times \N^{\N}\to \N$ satisfying 
\[\forall k\in \N \, \forall g:\N\to\N \, \exists N\leq \Phi(k,g)\, \forall i, j\in [N,N+g(N)]\, 
\left(\|x_i-x_j\|\leq \frac{1}{k+1}\right).\]
As metastability of a sequence is non-effectively equivalent with the Cauchy property,   an effective rate of metastability is the best quantitative information one can obtain when rates of convergence are not to be expected, as it is the case with the HPPA $(x_n)$. 
Metastability was used by Tao \cite{Tao08} and Walsh \cite{Wal12} to obtain far-reaching generalizations of the von Neumann mean ergodic theorem. 

By letting $g(n)=L\in\N$, we obtain a mapping $\Phi_L:\N\to \N$ such that 
\[\forall k\in \N \, \,\exists N\leq \Phi_L(k)\, \,\forall i, j\in [N,N+L]\, 
\left(\|x_i-x_j\|\leq \frac{1}{k+1}\right).\]
We call such a function $\Phi_L$  a rate of $L$-metastability of $(x_n)$. By varying $L$, we get that $(x_n)$ is stable on arbitrarily long time-intervals. 

The effective bounds for the HPPA $(x_n)$ are computed by using methods from proof mining  (a research field in applied proof theory), transforming the arguments from
Boikanyo and Moro{\c s}anu's proof of Theorem~\ref{main-theorem} into new ones, providing the computational information which was previously hidden.  We refer to Kohlenbach's book \cite{Koh08} for a comprehensive introduction to proof mining and to \cite{Koh19,Koh20} for surveys on more recent applications. 
The proof mining research program is strongly related to Tao’s proposal  \cite{Tao07} of ``hard" analysis, based on quantitative, finitary arguments, obtained by 
the conversion, the finitization of infinitary, qualitative arguments in different classes of proofs; as Tao points out in \cite{Tao07}, this can be obtained using methods from proof theory,  known as proof interpretations.

Quantitative results on Halpern-type proximal point algorithms have been only recently obtained by 
Kohlenbach \cite{Koh20a} and the second author \cite{Pin21}. However,  proof mining  has been applied in a series of papers to obtain rates 
of asymptotic regularity and metastability for the Halpern iteration \cite{Leu07a,Koh11,KohLeu12a,LeuNic16,KohSip21,FerLeuPin19} and  rates of metastability and convergence for the Proximal Point Algorithm \cite{KohLeuNic18,LeuNicSip18,LeuSip18,LeuSip18b,Koh21}. These rates are computed both for Hilbert spaces and for more general classes of spaces: uniformly convex (and uniformly smooth) Banach spaces,   $CAT(0)$ spaces and  CAT$(\kappa)$ spaces (with $\kappa >0$).

We recall in the following some quantitative notions. Let $(a_n)_{n\in\N}$ be a 
sequence in $H$. 
If $(a_n)$ converges to $a\in H$, then a rate of convergence for 
$(a_n)$ is a mapping $\gamma:\N\to\N$ such that
\[\forall k\in\N\, \forall n\geq \gamma(k)\,\left(\|a_n-a\|\leq \frac1{k+1}\right).\]
If $(a_n)$ is Cauchy, then a Cauchy modulus of $(a_n)$ is a mapping $\chi:\N\to\N$ satisfying
\[\forall k\in\N\,\forall n\in\N\,\left(\|a_{\chi(k)+n}-a_{\chi(k)}\|\le \frac1{k+1}\right).\]
As in the case of the Cauchy property, one has also a metastable version of the convergence of a 
sequence.  If $(a_n)$ converges to $a\in H$, we say, following \cite{Koh11}, that a quasi-rate of 
convergence of $(a_n)$ is a mapping $\widetilde{\gamma}:\N\times \N^\N\to \N$ satisfying
\[
\forall k\in\N\,\forall g:\N\to\N\,\forall n\in[\widetilde{\gamma}(k,g),\widetilde{\gamma}(k,g)+g(\widetilde{\gamma}(k,g))]\, 
\left(\|a_n-a\|\leq \frac1{k+1}\right).
\]
Obviously, if $\gamma$ is a rate of convergence of $(a_n)$, then  $\widetilde{\gamma}(k,g):=\gamma(k)$ (for all $k,g$) 
is a quasi-rate of convergence of $(a_n)$.

Let $(b_n)$ be a sequence of nonnegative real numbers. If the series $\ds\sum_{n=0}^\infty b_n$ 
diverges, then a function $\theta:\N\to\N$ is called a rate 
of divergence of the series if $\ds \sum_{i=0}^{\theta(n)} b_i \geq n$ for all $n\in\N$. 
If $\limn b_n=0$, then $\gamma$ is a rate of convergence of $(b_n)$ if and only if for any $k\in\N$, we have that 
$\ds b_n\leq \frac1{k+1}$  for all $n\geq \gamma(k)$.

\section{Some useful results on $(x_n)$}\label{useful-xn}

In this section, $H$ is a real Hilbert space, $A:H\to2^{H}$ is a maximally monotone operator 
with $zer(A)\neq \emptyset$, $J_{\gamma A}$ is the resolvent of $\gamma A$  ($\gamma>0$) and  the sequence $(x_n)$ is given by \eqref{HPPA-def-1}. The following well-known resolvent identity will be useful:  for any $\beta,\gamma>0$ and $x\in H$, 
\beq
J_{\beta A}(x)=J_{\gamma A}\left(\frac{\gamma}{\beta}x+\left(1-\frac{\gamma}{\beta}\right)J_{\beta A}(x)\right) \label{resolvent-id}
\eeq

\subsection{Upper bounds on $(x_n)$}

\blem
Let $p$ be a zero of $A$. Then, for all $n\in\N$,
\bea
\|x_{n+1}-p\| &\leq & \alpha_n\|u-p\| +(1-\alpha_n)\|x_n-p\|+\|e_n\|, \label{ineq-xn-1} \\
\|x_n-p\| &\leq & \max\{ \|u-p\|, \|x_0-p\| \}+\sum_{i=0}^{n-1}\|e_i\|. \label{ineq-xn-2}
\eea
\elem
\begin{proof}
We have that 
\bua
\|x_{n+1}-p\|&=& \|\alpha_n u+(1-\alpha_n)J_{\beta_nA}x_n+e_n-p\|\\
&=& \|\alpha_n(u-p)+(1-\alpha_n)(J_{\beta_nA}x_n-J_{\beta_nA}p)+e_n\| \\
&&  \text{since~} p \text{~is a fixed point of~} J_{\beta_nA} \\
&\leq & \alpha_n\|u-p\| +(1-\alpha_n)\|x_n-p\|+\|e_n\| \\
&&  \text{since~} J_{\beta_nA} \text{~is nonexpansive}.
\eua
Thus, \eqref{ineq-xn-1} holds. We obtain \eqref{ineq-xn-2} by an easy induction on $n$. 
\end{proof}

In \cite{BoiMor10} it is shown that the sequence $(x_n)$ is bounded if $\ds \sum_{n=0}^\infty \|e_n\|<\infty$ or 
the sequence $\left(\frac{\|e_n\|}{\alpha_n}\right)$ is
bounded. Next it's a quantitative version of this result.

\begin{lemma}\label{quant-xn-bounded}
Let $p$ be a zero of $A$ and $D$ be a natural number. Define
\[
D_1:=\max\{\|u-p\|,\|x_0-p\|\}+D, \quad D_2:=\max\{2(\|u-p\|+D), \|x_0-p\|\}.
\]
Then 
\be
\item\label{quant-xn-bounded-sum} if $D$ is an upper bound on $\left(\sum\limits_{i=0}^{n}\|e_i\|\right)$, then $D_1$ is an upper bound on
the sequence $\left(\|x_n-p\|\right)$.
\item\label{quant-xn-bounded-frac} if $D$ is an upper bound on $\left( \dfrac{\|e_n\|}{\alpha_n}\right)$, then $D_2$ is an upper bound on
the sequence $\left(\|x_n-p\|\right)$.
\ee 
\end{lemma}
\begin{proof}
\be
\item Apply \eqref{ineq-xn-2}.
\item The proof is by induction on $n$. The case $n=0$ is trivial.

$n\Ra n+1$: We get that 
\bua
\!\!\!\!\!\! \|x_{n+1}-p\|^2 &=& \left\| \alpha_n u+(1-\alpha_n)J_{\beta_nA}x_n+e_n-p\right\|^2\\
&=& \left\| \alpha_n\left(u-p+\dfrac{e_n}{\alpha_n}\right) +
(1-\alpha_n)\left(J_{\beta_nA}x_n-J_{\beta_nA}p\right)\right\|^2\\
&\leq & (1-\alpha_n)^2\|J_{\beta_nA}x_n-J_{\beta_nA}p\|^2
+2\alpha_n\!\left\langle \!u-p+\dfrac{e_n}{\alpha_n}, x_{n+1}-p \!\right\rangle \\
&&\text{since~}  \|x+y\|^2\leq \|y\|^2+ 2\langle x, x + y\rangle \text{~for all~}x,y\in H\\
&\leq &(1-\alpha_n)^2\|x_n-p\|^2+2\alpha_n\left( \|u-p\|+\dfrac{\|e_n\|}{\alpha_n} \right)\|x_{n+1}-p\| \\
&\leq & (1-\alpha_n)^2D_2^2+\alpha_n D_2 \|x_{n+1}-p\| \\
&& \text{by the induction hypothesis and the definition of~}D_2.
\eua
It follows that 
\[
	\left( \|x_{n+1}-p\|-\frac{D_2}2\alpha_n \right)^2\leq (1-\alpha_n)^2D_2^2 + \frac{D_2^2}4\alpha_n^2,
\]
hence
\bua 
\|x_{n+1}-p\|  &\leq &  \frac{D_2}2\alpha_n+\sqrt{(1-\alpha_n)^2D_2^2 + \frac{D_2^2}4\alpha_n^2}\\
&\leq &  \frac{D_2}2\alpha_n+(1-\alpha_n)D_2+\frac{D_2}2\alpha_n = D_2.
\eua
\ee
\end{proof}

\subsection{An approximate fixed point sequence}

One of the main ingredients of Boikanyo and Moro{\c s}anu's proof of Theorem~\ref{main-theorem} is
a classical theorem of Browder \cite{Bro67} on the strong convergence of a sequence of 
approximants to  fixed points of nonexpansive mappings. Kohlenbach \cite{Koh11} applied
proof mining methods both to Browder's original proof and to a simplified proof of this theorem, 
due to Halpern \cite{Hal67}.
In the sequel, we apply Kohlenbach's quantitative version of Browder's theorem obtained by 
the logical analysis of Halpern's proof. 

For each $n\in\N$, let us define 
\[S_n:H\to H, \quad S_n(x)=\alpha_nu+(1-\alpha_n)J_{\beta A} x,\]
where $\beta>0$. 

Then for $\alpha_n\in(0,1]$, $S_n$ is a contraction, hence, by the Banach contraction principle, $S_n$ has 
a unique fixed point $z_n$. Thus, 
\beq
z_n=\alpha_nu+(1-\alpha_n)J_{\beta A} z_n \quad \text{for all~}n\in\N. \label{def-zn}
\eeq

\blem
Let $p\in zer(A)$. Then, for all $n\in\N$, 
\bea
\|z_n-p\| & \leq & 2\|u-p\|, \label{zn-p-bounded}\\
\|z_n-u\| & \leq & 3\|u-p\|, \label{zn-u-bounded}\\
\|J_{\beta A} z_n-u\| & \leq & 3\|u-p\|. \label{Jzn-u-bounded}
\eea
\elem
\begin{proof} We have that 
\bua
\|z_n-p\|^2 &=& \|\alpha_n(u-p)+(1-\alpha_n)(J_{\beta A} z_n-p)\|^2 \\
&\leq & (1-\alpha_n)^2\|J_{\beta A} z_n-p\|^2+2\lan \alpha_n(u-p), z_n-p\ran\\
& = & (1-\alpha_n)^2\|J_{\beta A} z_n-J_{\beta A}p\|^2+2\lan \alpha_n(u-p), z_n-p\ran\\
&\leq &(1-\alpha_n)^2 \|z_n-p\|^2+2\alpha_n\|u-p\|\|z_n-p\|.
\eua
We used above the fact that $\|x+y\|^2\leq \|y\|^2+ 2\langle x, x + y\rangle$ for all $x,y\in H$.
It follows that
\bua
\alpha_n(2-\alpha_n)\|z_n-p\| \leq 2\alpha_n\|u-p\|,
\eua
hence \eqref{zn-p-bounded}. One obtains immediately \eqref{zn-u-bounded}.  
To get \eqref{Jzn-u-bounded}, remark that 
\[\|J_{\beta A} z_n-u\|\leq \|J_{\beta A} z_n-p\|+\|u-p\|\leq \|z_n-p\|+\|u-p\|.\]
\end{proof}

The inequalities below are obtained in the proof of \cite[Theorem 2]{BoiMor11}.

\blem
For all $n\in\N$,
\bea
 \|x_{n+1}-z_n\| &\leq & (1-\alpha_n)\|x_n-z_n\|+
\dfrac{\alpha_n|\beta\!-\!\beta_n|}{\beta}\|u\!-\!J_{\beta A}z_n\|+\|e_n\|,
\label{ineq-zn-xn-1}\\
\|z_n-z_{n+1}\| &\leq &\frac{|\alpha_n-\alpha_{n+1}|}{\alpha_n}\|u\!-\!J_{\beta A}z_{n+1}\|.\label{ineq-zn-zn1}
\eea
\elem

The following quantitative result on the behaviour of $(z_n)$  is a special case of \cite[Theorem~4.2]{Koh11}. 

\bprop\label{main-quant-zn}
Let $d\in\N^*$ be such that $d\geq 3\|u-p\|$ for some zero $p$ of $A$. 
\be
\item\label{main-quant-zn-alphan-noninc} Assume that $(\alpha_n)$ is a nonincreasing sequence. Then
$(z_n)$ is Cauchy with rate of metastability $\Omega_{d}$, given by
\beq
\Omega_{d}(k,g) := \tilde{g}^{(d^2(k+1)^2)}(0),
\eeq
with $\tilde{g}(n):=n+g(n)$.
\item\label{main-quant-zn-alphan-zero} Assume that $\limn \alpha_n=0$ with quasi-rate of convergence $\chi$ and let $h:\N\to\N$ be such 
that $\alpha_n\geq \frac1{h(n)+1}$ for all $n\in\N$.
Then $(z_n)$ is Cauchy with rate of metastability $\widetilde{\Omega}_{\chi,h,d}$, given by
\beq
\widetilde{\Omega}_{\chi,h,d}(k,g) := \chi_g^M\left(g_{h,\chi_g}^{(4d^2(k+1)^2)}(0)\right)
\eeq
with $\chi_g(n) := \chi(n,g)$, $\chi_g^M(n) := \max\{\chi_g(i)\mid i\leq n\}$ and 
$g_{h,\chi_g}(n) :=	 \max\{h(i)\mid i\leq \chi_g(n)+g(\chi_g(n))\}$.
\ee
\eprop
\begin{proof}
Apply \cite[Theorem~4.2]{Koh11} with $v_0:=u$, $U:=J_{\beta A}$, $s_n:=1-\alpha_n$, 
$\eps:=\frac1{k+1}$ and $d$, $h$, $\chi_g$ as above.
\end{proof}

\section{Quantitative lemmas on sequences of real numbers}

In the sequel, $(a_n)_{n\in\N}$ is a sequence in $[0,1]$, $(b_n)_{n\in\N}$ is a sequence of real numbers and 
$(c_n)_{n\in\N}, (s_n)_{n\in\N}$ are sequences of nonnegative real numbers satisfying, for all $n\in\N$,
\beq
s_{n+1}\leq (1-a_n)s_n + a_n b_n+c_n. \label{def-sn-cn}
\eeq

The following lemma from \cite{Xu02} is one of the main tools used in the proof of Theorem~\ref{main-theorem}.

\blem\label{lemma-Xu02}
Assume that $\sum\limits_{n=0}^\infty a_n$ diverges (or equivalently, $\ds\prod_{n=0}^\infty (1-a_n)=0$ if $a_n<1$ for all $n\in\N$), $\lsupn b_n \leq 0$ 
and $\sum\limits_{n=0}^\infty c_n$ converges. Then $\limn s_n =0$.
\elem

We give in the sequel quantitative versions of Lemma~\ref{lemma-Xu02}. We remark that for a particular 
case of this lemma, obtained by letting $c_n:=0$, the first author and Kohlenbach have already 
proved quantitative versions in \cite{KohLeu12a}.  The proofs of the following results are similar with those 
of \cite[Lemmas 5.2, 5.3]{KohLeu12a}. However, for the sake of completeness, we give them in this paper.

\blem\label{quant-Xu02-lem1}
Let $p,N\in\N$ be such that 
\beq \label{quant-Xu02-lem1-hyp}
b_n\leq \frac1{p+1} \quad \text{for all~}n\geq N. 
\eeq 
Then for all $m,n\in\N$  with $n\geq N$,
\[
s_{n+m+1}\leq \left( \prod_{i=n}^{n+m}(1-a_i) \right)s_n + 
\left( 1-\prod_{i=n}^{n+m}(1-a_i) \right)\frac{1}{p+1}+\sum\limits_{i=n}^{n+m}c_i.
\]
\elem
\begin{proof}
The proof is by induction on $m$. The case $m=0$ is trivial.\\
$m\Ra m+1$:  For simplicity, let us denote 
\bua
A:=\prod_{i=n}^{n+m}(1-a_i).
\eua
We get that
\bua 
s_{n+m+2} &\leq & (1-a_{n+m+1})s_{n+m+1}+ a_{n+m+1} b_{n+m+1} + c_{n+m+1} \\
&\leq & (1-a_{n+m+1})\left[ As_n + 
\left( 1-A \right)\frac{1}{p+1}+\sum\limits_{i=n}^{n+m}c_i \right] + \, a_{n+m+1}\frac{1}{p+1}\\
&&  +c_{n+m+1} \qquad \text{by the induction hypothesis and \eqref{quant-Xu02-lem1-hyp}} \\
& = & \left( \prod_{i=n}^{n+m+1}(1-a_i) \right)s_n+
\left((1-a_{n+m+1})\left( 1-A\right)+a_{n+m+1}\right)\frac{1}{p+1}\\
&&+ (1-a_{n+m+1})\sum\limits_{i=n}^{n+m}c_i+c_{n+m+1}\\
&\leq &  \left( \prod_{i=n}^{n+m+1}(1-a_i) \right)s_n + \left( 1-\prod_{i=n}^{n+m+1}(1-a_i) \right)\frac{1}{p+1}+
\sum\limits_{i=n}^{n+m+1}c_i.
\eua
\end{proof}

\blem\label{quant-lem-Xu02-lem2}
Let $M\in\N^*$ and $\psi, \chi:\N\to\N$ be such that 
\be
\item\label{M-ub-sn}  $M$ is an upper bound on $(s_n)$;
\item\label{psi-bn} for all $k\in\N$ and all $n\in\N$  with $n\geq \psi(k)$, $b_n \leq \frac{1}{k+1}$;
\item\label{chi-cn} $\chi$ is a Cauchy modulus for $\left(\tilde{c}_n:=\sum\limits_{i=0}^n c_i\right)$.
\ee
Define 
\beq \label{def-lsn}
\lsn:\N\to\N, \quad \lsn(k):=\max\{\psi(3k+2), \chi(3k+2)+1\}.
\eeq
Then for all $k,m\in\N$ and all $n\in\N$  with $n\geq \lsn(k)$,
\[s_{n+m+1}\leq M \prod_{i=n}^{n+m} (1-a_i)+\frac{2}{3(k+1)}.\]
\elem
\begin{proof}
Let $k, m,n\in\N$ be such that $n\geq \lsn(k)$. Since $\lsn(k)\geq \psi(3k+2)$, we get from \eqref{psi-bn} that
$b_n\leq \frac{1}{3(k+1)}$.  We apply Lemma~\ref{quant-Xu02-lem1} to obtain that
\bua
s_{n+m+1} &\leq & \left( \prod_{i=n}^{n+m}(1-a_i) \right)s_n + 
\left( 1-\prod_{i=n}^{n+m}(1-a_i) \right)\frac{1}{3(k+1)}+\sum\limits_{i=n}^{n+m}c_i\\
&\leq &  M\prod_{i=n}^{n+m}(1-a_i)+\frac{1}{3(k+1)}+\sum\limits_{i=n}^{n+m}c_i \quad \text{by \eqref{M-ub-sn}}.
\eua
As $n\geq \lsn(k)\geq \chi(3k+2)+1$, we apply \eqref{chi-cn} to get, 
by letting $r:=n+m-\chi(3k+2)$, that
\bua
\sum\limits_{i=n}^{n+m}c_i &\leq & \sum\limits_{i=\chi(3k+2)+1}^{n+m}c_i = 
\sum\limits_{i=0}^{n+m}c_i- 
\sum\limits_{i=0}^{\chi(3k+2)}c_i = \tilde{c}_{\chi(3k+2)+r}-\tilde{c}_{\chi(3k+2)}\\
&\leq & \frac{1}{3(k+1)}.
\eua
The conclusion follows. 
\end{proof}

\bprop\label{quant-lem-Xu02-prop1}
In the hypothesis of Lemma~\ref{quant-lem-Xu02-lem2}, assume, moreover, that $\sum\limits_{n=0}^\infty a_n$  diverges with rate of 
divergence $\theta:\N\to\N$. Define $\Sigma:=\Sigma_{M,\theta,\psi,\chi}$ by 
\bea
\Sigma:\N\to\N, \quad \Sigma(k)=\theta(\lsn(k)+\lceil \ln(3M(k+1))\rceil)+1,
\eea
where $\lsn$ is given by \eqref{def-lsn}. 

Then $\limn s_n=0$ with rate of convergence $\Sigma$. 
\eprop
\begin{proof}
Let $k\in\N$ be arbitrary. Denote
\[K:=\Sigma(k)-\lsn(k)-1=\theta(\lsn(k)+\lceil \ln(3M(k+1))\rceil)-\lsn(k).\]
As $a_n \leq 1$,  we have that $\theta(n+1)\geq n$ for all $n$.  Thus, $K\in\N$. 
For all $m\geq K$, we get that
\bua
\sum\limits_{i=\lsn(k)}^{\lsn(k)+m}a_i & \geq & \sum\limits_{i=\lsn(k)}^{\lsn(k)+K} a_i =
\sum\limits_{i=0}^{\theta(\lsn(k)+\lceil \ln(3M(k+1))\rceil)}a_i - \sum\limits_{i=0}^{\lsn(k)-1}a_i \\
		& \geq & \lsn(k)+\lceil \ln(3M(k+1))\rceil - \sum\limits_{i=0}^{\lsn(k)-1}a_i \\
		& \geq &   \ln (3M(k+1)),
\eua
hence, 
\bua
\prod_{i=\lsn(k)}^{\lsn(k)+m} (1-a_i)& \leq & \exp \left( -\sum\limits_{i=\lsn(k)}^{\lsn(k)+m}a_i\right)
 \quad\text{since~} 1-x\leq \exp(-x) \text{~for~}x\geq 0\\
& \leq & \dfrac{1}{3M(k+1)}.
\eua
Let $n\geq \Sigma(k)$. Applying Lemma~\ref{quant-lem-Xu02-lem2} with $m:=n-\lsn(k)-1\geq K$, it follows that
\bua
s_n &=& s_{\lsn(k)+m+1} \leq M\prod_{i=\lsn(k)}^{\lsn(k)+m} (1-a_i)+\frac{2}{3(k+1)}\leq \frac1{k+1}.
\eua
\end{proof}

In the sequel, we give another quantitative version of Lemma~\ref{lemma-Xu02}, when the hypothesis  
$\prod\limits_{n=0}^\infty (1-a_n)=0$ is used. Let us denote $P_n:=\prod\limits_{j=0}^n(1-a_j)$ for all $n\in\N$. Then 
$\prod\limits_{n=0}^\infty (1-a_n)=0$ if and only if $\limn P_n=0$. A rate of convergence of 
$\prod\limits_{n=0}^\infty (1-a_n)$ 
towards $0$ will be a rate of convergence of $(P_n)$ towards $0$.

\bprop\label{quant-lem-Xu02-prop2}
In the hypothesis of Lemma~\ref{quant-lem-Xu02-lem2}, assume, furthermore, that 
\be
\item\label{quant-lem-Xu02-prop2-less-1} $a_n<1$ for all $n\in\N$;
\item\label{quant-lem-Xu02-prop2-prod} $\prod\limits_{n=0}^\infty (1-a_n)=0$ with rate of convergence $\theta$;
\item\label{quant-lem-Xu02-prop2-N0} $\lsn_0:\N\to\N^*$ is such that for all $k\in\N$,
$\frac1{\lsn_0(k)}\leq P_{\lsn(k)-1}$.
\ee
Define $\widetilde{\Sigma}:=\widetilde{\Sigma}_{M,\theta,\psi,\chi,\lsn_0}$ by 
\bea
\widetilde{\Sigma}:\N\to\N, \quad \widetilde{\Sigma}(k)=
\max\left\{\theta\left(3M\lsn_0(k)(k+1)-1\right),\lsn(k)\right\}+1.
\eea
Then $\limn s_n=0$ with rate of convergence $\widetilde{\Sigma}$.
\eprop
\begin{proof}
Let $k\in\N$ and $n\geq \widetilde{\Sigma}(k)$. Applying Lemma~\ref{quant-lem-Xu02-lem2} with $m:=n-\lsn(k)-1\in\N$, 
we get that
\bua
s_n &=& s_{\lsn(k)+m+1} \leq M \prod\limits_{i=\lsn(k)}^{\lsn(k)+m} (1-a_i)+\frac{2}{3(k+1)} = 
\frac{MP_{\lsn(k)+m}}{P_{\lsn(k)-1}}+\frac{2}{3(k+1)}\\
&\leq & M\lsn_0(k) P_{\lsn(k)+m}+\frac{2}{3(k+1)} \quad\text{by \eqref{quant-lem-Xu02-prop2-N0}}\\
&\leq & \frac1{k+1},
\eua
as $\lsn(k)+m=n-1\geq \theta(3M\lsn_0(k)(k+1)-1)$, hence, by \eqref{quant-lem-Xu02-prop2-prod}, 
$P_{\lsn(k)+m}\leq \frac1{3M\lsn_0(k)(k+1)}$.
\end{proof}

In the proof of our second main theorem, we shall need a particular case of the inequality \eqref{def-sn-cn}, 
obtained by letting $c_n:=0$ for all $n\in\N$. By an easy adaptation of the previous proofs, 
we obtain the following quantitative result.

\bprop\label{quant-lem-Xu02-cn-0}
Let $(a_n)_{n\in\N}$ be a sequence in $(0,1)$, $(b_n)_{n\in\N}$ be a sequence of real numbers and 
$(s_n)_{n\in\N}$ be a sequence of nonnegative real numbers satisfying, for all $n\in\N$,
\beq
s_{n+1}\leq (1-a_n)s_n + a_n b_n. \label{def-sn}
\eeq
Assume that $M\in\N^*$ is an upper bound on $(s_n)$ and that $\psi^*:\N\to\N$ is such that $b_n \leq \frac{1}{p+1}$
for all $p,n\in\N$  with $n\geq \psi^*(p)$. Define
$$\lsn^*:\N\to\N, \quad \lsn^*(k)=\psi^*(2k+1).$$

The following hold:
\be
\item\label{quant-lem-Xu02-cn-0-C1q} If $\sum\limits_{n=0}^\infty a_n$  diverges with rate of 
divergence $\theta$, then $\limn s_n=0$ with rate of convergence 
$$\Sigma^*(k)=\theta(\lsn^*(k)+\lceil \ln(2M(k+1))\rceil)+1.$$
\item\label{quant-lem-Xu02-cn-0-C2q} If $\prod\limits_{n=0}^\infty (1-a_n)=0$ with rate of convergence $\theta$ and $\lsn^*_0:\N\to\N^*$ is such that 
$\ds \frac1{\lsn^*_0(k)}\leq P_{\lsn^*(k)-1}$ for all $k\in\N$, then $\limn s_n=0$ with rate of convergence 
$$\widetilde{\Sigma^*}(k)=
\max\left\{\theta\left(2M\lsn^*_0(k)(k+1)-1\right),\lsn^*(k)\right\}+1.$$
\ee
\eprop

\section{Main results}\label{section-main}

In this section we  compute uniform effective bounds  on the asymptotic behaviour of the HPPA $(x_n)$  and of the approximate fixed 
point sequence $(z_n)$. These culminate with quantitative versions of Theorem~\ref{main-theorem}, 
providing uniform effective rates of metastability for the HPPA $(x_n)$.  The fact that one can obtain such effective bounds  
is guaranteed by general logical metatheorems for Hilbert spaces proved by Kohlenbach \cite{Koh05}.

We need quantitative versions of the hypotheses (C0) - (C6) of Theorem~\ref{main-theorem}:\\

\bt{lll}
$(C0_{q})$ &  $\ds\limn\alpha_n=0$ with rate of convergence $\rateczeroq$, \\[2mm]
$(C0^*_{q})$ &  $\ds\limn\alpha_n=0$ with quasi-rate of convergence $\rateczeroqs$, \\[2mm]
$(C1_q)$ & $\ds\sum\limits_{n=0}^\infty \alpha_n$ diverges with rate of divergence $\rateconeq$,\\[2mm]
$(C2_q)$ & $\ds\prod_{n=0}^\infty (1-a_n)=0$ with rate of convergence $\ratectwoq$, \\[2mm]
$(C3_q)$ & $\ds\limn \frac{|\alpha_{n+1}-\alpha_n|}{\alpha_n^2} =0$ with rate of convergence $\ratecthreeq$, \\[2mm]
$(C4_q)$ &  $\ds\limn\beta_n=\beta>0$ with rate of convergence $\ratecfourq$, \\[2mm]
$(C5_q)$ & $\ds\sum\limits_{n=0}^\infty \|e_n\|<\infty$ with Cauchy modulus $\ratecfiveq$, \\[2mm]
$(C6_q)$ & $\ds\limn\frac{\|e_n\|}{\alpha_n}=0$ with rate of convergence $\ratecsixq$.\\[3mm]
\et

We assume for the rest of this section that $H$ is a Hilbert space, $A:H\to 2^H$ is a maximally monotone operator such that $zer(A)\ne\emptyset$, $(x_n)$ is defined by 
\eqref{HPPA-def-1} and $b\in\N^*$  is such that 
\beq
b\geq \max\{\|x_0-p\|, \ \|u-p\|\} \text{~for some zero~} p \text{~of~} A. \label{b-hyp-main} 
\eeq
We suppose, moreover, that  $(C4_q)$ holds and $(z_n)$ is defined by \eqref{def-zn} with $\beta=\limn \beta_n$. 

\subsection{Rates of convergence for $(\|x_n-z_n\|)$}

One of the main steps in the proof of Theorem~\ref{main-theorem} is to obtain that 
\beq \limn\|x_n-z_n\|=0.  \label{xn-zn-0}
\eeq 
In the sequel we give quantitative versions of \eqref{xn-zn-0}, consisting of uniform effective 
rates of convergence.

The first quantitative result is the following.

\bprop\label{lim-xn-zn-0-v1}
Assume that $(C3_q)$, $(C5_q)$ hold and $\ell\in\N, D\in\N^*$   satisfy 
\beq
\beta\geq \frac{1}{\ell +1}, \quad D\geq \sum\limits_{i=0}^{\ratecfiveq(0)}\|e_i\|+1. \label{def-ell-D}
\eeq
Define
\bua
\psi(k) &:=& \max\{ \ratecfourq(6b(\ell +1)(k+1)-1),\ratecthreeq(6b(k+1)-1)\}, \\ 
\lsn(k) &:=& \max\{\psi(3k+2), \ratecfiveq(3k+2)+1\}. 
\eua
The following hold:
\be
\item\label{lim-xn-zn-0-v1-C1q} If $(C1_q)$ holds, then $\limn \|x_n-z_n\|=0$ with rate of convergence 
$\Theta$ given  by
\beq
\Theta(k) := \rateconeq(\lsn(k)+\lceil \ln(3(D+5b)(k+1))\rceil)+1. \label{def-Theta-main}
\eeq 
\item\label{lim-xn-zn-0-v1-C2q}  If $(C2_q)$ holds and $\lsn_0:\N\to\N^*$ is such that 
\beq\frac1{\lsn_0(k)}\leq \prod\limits_{j=0}^{\lsn(k)-1}(1-\alpha_j) \text{~for all~} k\in\N, \label{hypothesis-lsn0}
\eeq
then  $\limn \|x_n-z_n\|=0$ with rate of convergence 
$\widetilde{\Theta}$ defined by 
\beq
\widetilde{\Theta}(k) := \max\left\{\ratectwoq\left(3(D+5b)\lsn_0(k)(k+1)-1\right),\lsn(k)\right\}+1. \label{def-wTheta-main}
\eeq
\ee
\eprop
\begin{proof}
By \eqref{ineq-zn-xn-1}, \eqref{ineq-zn-zn1}, \eqref{Jzn-u-bounded} and the hypothesis on $b$, we get that for all $n\in\N$,
\bua
\|x_{n+1}-z_{n+1}\| &\leq & \|x_{n+1}-z_n\|+\|z_n-z_{n+1}\| \\
&\leq &  (1-\alpha_n)\|x_n-z_n\|+\alpha_nb_n+\|e_n\|,
\eua
where $\ds b_n:= 3b\left(\frac{|\beta-\beta_n|}{\beta}+\frac{|\alpha_n-\alpha_{n+1}|}{\alpha_n^2}\right)$. 
We verify in the sequel that, by letting 
\bce
$s_n:=\|x_n-z_n\|$, $a_n:=\alpha_n$, $b_n$ as above and $c_n:=\|e_n\|$, 
\ece
the hypotheses of Lemma~\ref{quant-lem-Xu02-lem2} are satisfied.

Applying $(C5q)$, one can easily see that, for all $n\in\N$, 
$\sum\limits_{i=0}^{n}\|e_i\|\leq 1+\sum\limits_{i=0}^{\ratecfiveq(0)}\|e_i\|\leq D$.
Hence, $D$ is an upper bound on $\left(\sum\limits_{i=0}^{n}\|e_i\|\right)$. \\
Let $p\in zer(A)$ satisfy the hypothesis on $b$. We get that 
\bua
\|x_n-z_n\|&\leq & \|x_n-p\|+\|p-u\|+\|u-z_n\| \\
&\leq & D+5b \quad \text{by Lemma~\ref{quant-xn-bounded}.\eqref{quant-xn-bounded-sum}
and \eqref{zn-u-bounded}}.
\eua

Let $n\geq\psi(k)$ be arbitrary. Applying $(C3_q)$, $(C4_q)$ and the hypothesis on $\ell$, 
we get that 
\bua
\frac{|\alpha_n-\alpha_{n+1}|}{\alpha_n^2} \leq  \frac{1}{6b(k+1)} \quad \text{and} \quad 
\frac{|\beta-\beta_n|}{\beta} \leq  (\ell+1)|\beta-\beta_n|\leq \frac{1}{6b(k+1)}.
\eua
It follows that $b_n \leq \frac{1}{k+1}$ for all $n\geq\psi(k)$. 

Finally, by $(C5_q)$, we have that $\ratecfiveq$ is a Cauchy modulus for $\sum\limits_{n=0}^\infty \|e_n\|$.

Thus, the hypotheses of Lemma~\ref{quant-lem-Xu02-lem2} hold. We get \eqref{lim-xn-zn-0-v1-C1q} 
by applying Proposition~\ref{quant-lem-Xu02-prop1} 
with $M:=D+5b$, $\chi:=\ratecfiveq$, $\theta:=\rateconeq$, 
and $\psi, \lsn$ as above. Furthermore, \eqref{lim-xn-zn-0-v1-C2q} is obtained by applying 
Proposition~\ref{quant-lem-Xu02-prop2} with $M:=D+5b$, $\chi:=\ratecfiveq$, $\theta:=\ratectwoq$, 
and $\psi, \lsn, \lsn_0$ as above.
\end{proof}

A result similar to Proposition~\ref{lim-xn-zn-0-v1} can be obtained by replacing, in the hypothesis, 
$(C5_q)$ with $(C6_q)$.

\bprop\label{lim-xn-zn-0-v2}
Assume that $(C3_q)$, $(C6_q)$ hold and $\ell\in\N, D^*\in\N^*$   satisfy 
\beq
\beta\geq \frac{1}{\ell +1}, \quad D^* \geq \max_{i\leq\ratecsixq(0)}\left\{\frac{\|e_i\|}{\alpha_i},1\right\}. \label{def-ell-Dstar}
\eeq
Define
\bua
\psi^*(k) &:=& \max\{\ratecfourq(9b(\ell+1)(k+1)-1),\ratecthreeq(9b(k+1)-1), \ratecsixq(3k+2)\}, \\
\lsn^*(k) &:=& \psi^*(2k+1).
\eua
The following hold:
\be
\item\label{lim-xn-zn-0-v2-C1q} If $(C1_q)$ holds, then $\limn \|x_n-z_n\|=0$ with rate of convergence 
$\Theta^*$ defined by
\beq
\Theta^*(k) := \rateconeq(\lsn^*(k)+\lceil \ln(2(2D^*+6b)(k+1))\rceil)+1. \label{def-Theta-star-main}
\eeq
\item\label{lim-xn-zn-0-v2-C2q} If $(C2_q)$ holds and and $\lsn^*_0:\N\to\N^*$ is such that
\beq
\frac1{\lsn^*_0(k)}\leq \prod\limits_{j=0}^{\lsn^*(k)-1}(1-\alpha_j) \text{~for all~}k\in\N, \label{hypothesis-lsnstar0}
\eeq 
then 
$\limn \|x_n-z_n\|=0$ with rate of convergence $\widetilde{\Theta^*}$ defined by 
\beq
\widetilde{\Theta^*}(k) := \max\left\{\ratectwoq\left(2(2D^*+6b)\lsn^*_0(k)(k+1)-1\right),\lsn^*(k)\right\}+1.
\label{def-wTheta-star-main}
\eeq
\ee
\eprop
\begin{proof}
As in the proof of Proposition~\ref{lim-xn-zn-0-v2}, we apply \eqref{ineq-zn-xn-1}, \eqref{ineq-zn-zn1}, 
\eqref{Jzn-u-bounded} and the hypothesis on $b$ to obtain, for all $n\in \N$,
\bua
\|x_{n+1}-z_{n+1}\| &\leq & \|x_{n+1}-z_n\|+\|z_n-z_{n+1}\| \leq  (1-\alpha_n)\|x_n-z_n\|+\alpha_nb_n,
\eua
where $$\ds b_n:= 3b\left(\frac{|\beta-\beta_n|}{\beta}+\frac{|\alpha_n-\alpha_{n+1}|}{\alpha_n^2}\right)+
\frac{\|e_n\|}{\alpha_n}.$$
We shall apply Proposition~\ref{quant-lem-Xu02-cn-0} with  $s_n:=\|x_n-z_n\|$, $a_n:=\alpha_n$, $b_n$ as above.

By the hypothesis on $D^*$ and $(C6_q)$, one can see immediately that $D^*$ is an upper bound on  
$\left(\frac{\|e_n\|}{\alpha_n}\right)$.

If $p$ is a zero of $A$ as in the hypothesis, we get that 
\bua
\|x_n-z_n\|&\leq & \|x_n-p\|+\|p-u\|+\|u-z_n\| \\
&\leq & \|x_n-p\|+4b \quad \text{by \eqref{zn-u-bounded}}\\
&\leq & 2(b+D^*)+4b  \quad  \text{by Lemma~\ref{quant-xn-bounded}.\eqref{quant-xn-bounded-frac}}\\
&=& 2D^*+6b.
\eua
Let $n\geq \psi^*(k)$. We get that
\bua
b_n &=& 3b\left(\frac{|\beta-\beta_n|}{\beta}+\frac{|\alpha_n-\alpha_{n+1}|}{\alpha_n^2}\right)+
\frac{\|e_n\|}{\alpha_n}\\
&\leq & \frac{2}{3(k+1)}+ \frac{\|e_n\|}{\alpha_n} \quad \text{by the definition of~} \psi^*, (C3_q) 
\text{~and~} (C4_q)\\
&\leq & \frac{1}{k+1} \quad \text{by~} (C6_q) \text{~and the fact that~} \psi^*(k)\geq \ratecsixq(3k+2).
\eua
Thus, we can can apply Proposition~\ref{quant-lem-Xu02-cn-0}.\eqref{quant-lem-Xu02-cn-0-C1q}  with 
$M:=2D^*+6b$, $\theta:=\rateconeq$ and $\psi^*, \lsn^*$ as above to get \eqref{lim-xn-zn-0-v2-C1q}, and 
Proposition~\ref{quant-lem-Xu02-cn-0}.\eqref{quant-lem-Xu02-cn-0-C2q} with 
$M:=2D^*+6b$, $\theta:=\ratectwoq$ and $\psi^*, \lsn^*$, $\lsn_0^*$ as above to prove \eqref{lim-xn-zn-0-v2-C2q}.
\end{proof}

\subsection{Rates of asymptotic regularity}

One of the most useful concepts in metric fixed point theory and 
convex optimization is the asymptotic regularity \cite{BroPet66,BorReiSha92}: if $T:H\to H$ is a mapping and 
$(x_n)$ is a sequence in $H$, then $(x_n)$ is asymptotically regular with respect to $T$ if $\limn \|x_n-Tx_n\|=0$.
We can extend this notion to countable families of mappings: $(x_n)$ is asymptotically regular with respect to a
a family $(T_n:H\to H)_{n\in\N}$ of mappings if $\limn \|x_n-T_nx_n\|=0$. 
Rates of convergence of $(\|x_n-Tx_n\|)$, $(\|x_n-T_nx_n\|)$ towards $0$ are said to be rates of asymptotic regularity.

The following result shows that rates of asymptotic regularity can be computed, in the presence of $(C0_q)$,
from rates of convergence of the sequence $(\|x_n-z_n\|)$.

\bthm\label{main-asympt-reg}
Assume that $(C0_q)$ holds, $\ell\in \N^*$ is such that $\beta\geq \frac{1}{\ell+1}$ and 
that $\limn \|x_n-z_n\|=0$ with rate of convergence $\Lambda$. 

Then
\be
\item $(x_n)$ is asymptotically regular with respect to $J_{\beta A}$ with rate of asymptotic regularity 
$\Sigma$ defined by 
\[\Sigma(k) := \max\{\rateczeroq(6b(k+1)-1), \Lambda(4k+3)\}.\]
\item $(x_n)$ is asymptotically regular with respect to $(J_{\beta_n A})$ with rate of asymptotic regularity  $\Sigma_2$ defined by
\[\Sigma^*(k) := \max\{\ratecfourq(\ell),\Sigma(2k+1)\}.\]
\item  For every $i\in\N$, $(x_n)$ is asymptotically regular with respect to $J_{\beta_i A}$ with rate of asymptotic regularity 
$\Sigma_i$ defined by 
\[\Sigma_i(k) := \Sigma((1+(\ell+1)M_i)(k+1)-1),\]
where $M_i\in\N$ is such that $M_i\geq |\beta-\beta_i|$.
\ee
\ethm
\begin{proof}
\be
\item Let $k\in\N$ be arbitrary. For $n\geq \rateczeroq(6b(k+1)-1)$, using \eqref{Jzn-u-bounded} and 
the fact that $b$ satisfies \eqref{b-hyp-main}, we get that
\begin{equation*}
\|z_n-J_{\beta A}z_n\|=\alpha_n\|u-J_{\beta A}z_n\|\leq \alpha_n\cdot 3b\leq\frac{1}{2(k+1)}.
\end{equation*}
On the other hand, for $n\geq \Lambda(4k+3)$, we have that $\|x_n-z_n\|\leq \frac{1}{4(k+1)}$. 
It follows that for all $n\geq \Sigma(k)$,
\bua
\|x_n-J_{\beta A}x_n\|&\leq & \|x_n-z_n\|+\|z_n-J_{\beta A}z_n\|+\|J_{\beta A}z_n-J_{\beta A}x_n\|\\
&\leq & 2\|x_n-z_n\|+\|z_n-J_{\beta A}z_n\|\\
&\leq &  \frac{2}{4(k+1)}+\frac{1}{2(k+1)}=\frac{1}{k+1}.
\eua
\item We first note that for all $n\in\N$, 
\bua
\|J_{\beta A}x_n-J_{\beta_n A}x_n\| &\stackrel{\eqref{resolvent-id}}{=}& \left\|J_{\beta_n A}\left(\frac{\beta_n}{\beta}x_n+
\left(1-\frac{\beta_n}{\beta}\right)J_{\beta A}x_n\right)
-J_{\beta_n A}x_n\right\|\\
&\leq &  \left|1-\frac{\beta_n}{\beta}\right|\|x_n-J_{\beta A}x_n\|\\
&\leq &   (\ell+1)|\beta-\beta_n| \|x_n-J_{\beta A}x_n\|.
\eua
Let $k\in\N$ be arbitrary and $n\geq \Sigma^*(k)$. It follows that
\bua
\|x_n-J_{\beta_nA}x_n\| & \leq & \|x_n-J_{\beta A}x_n\|+\|J_{\beta A}x_n-J_{\beta_n A}x_n\| \\
& \leq & \|x_n-J_{\beta A}x_n\| + (\ell+1)|\beta-\beta_n|\|x_n-J_{\beta A}x_n\|\\
& \leq &  2 \|x_n-J_{\beta A}x_n\| \quad \text{since~} n \geq \ratecfourq(\ell), \text{~so~} |\beta-\beta_n|\leq \frac1{\ell+1}\\
& \leq & \frac1{k+1} \quad \text{since~} n\geq \Sigma(2k+1).
\eua
\item The proof is similar with the one of (ii).  Let $k\in\N$ and $n\geq \Sigma_i(k)$. Then 
\bua
\|x_n-J_{\beta_iA}x_n\| & \leq & \|x_n-J_{\beta A}x_n\|+\|J_{\beta A}x_n-J_{\beta_i A}x_n\| \\
& \leq & \|x_n-J_{\beta A}x_n\| + \left|1-\frac{\beta_i}{\beta}\right|\|x_n-J_{\beta A}x_n\|\\
& \leq & \|x_n-J_{\beta A}x_n\| + (\ell+1)M_i\|x_n-J_{\beta A}x_n\|\\
&=& (1+(\ell+1)M_i)\|x_n-J_{\beta A}x_n\|\\
& \leq & \frac1{k+1} \quad \text{since~} n\geq \Sigma((1+(\ell+1)M_i)(k+1)-1).
\eua
\ee
\end{proof}

\subsection{Quantitative versions of Theorem~\ref{main-theorem}}

We prove first a useful general result. 

\bprop\label{lemma-meta+conv-meta}
Let $(u_n)_{n\in \N}, (v_n)_{n\in \N}$ be sequences in $H$ such that 
\be
\item $(u_n)$ is Cauchy with rate of metastability $\Omega$;
\item $\limn \|u_n-v_n\|=0$ with rate of convergence $\vp$.
\ee
Then $(v_n)$ is Cauchy with rate of metastability $\Gamma$ given by 
\[
\Gamma(k,g)=\max\{\vp(3k+2),\ \Omega(3k+2,\widehat{g}_k)\},
\]
with $\widehat{g}_k(n):=\max\{\vp(3k+2), n\}-n+g(\max\{\vp(3k+2), n\})$.
\eprop
\begin{proof}
First, let us remark that for all $m,n\in\N$,
\beq\label{lemma-meta-1}
\|v_m-v_n\|\leq \|v_m-u_m\|+\|u_m-u_n\|+\|u_n-v_n\|.
\eeq
Let $k\in\N$ and $g:\N\to\N$ be arbitrary. Since $\Omega$  is a rate of metastability for $(u_n)$, 
there exists $N_0\leq \Omega(3k+2,\widehat{g}_k)$ such that
\[
\forall i, j\in [N_0,N_0+\widehat{g}_k(N_0)]\, \left(\|u_i-u_j\|\leq \frac{1}{3(k+1)}\right).
\]
Define $$N:=\max\{\vp(3k+2), N_0\}.$$
Since $N\geq N_0$ and $N_0+\widehat{g}_k(N_0)=N+g(N)$, we 
have that $[N, N+g(N)]\subseteq [N_0,N_0+\widehat{g}_k(N_0)]$. 
Hence,
\beq\label{lemma-meta-2}
\forall i, j\in [N, N+g(N)]\, \left(\|u_i-u_j\|\leq \frac{1}{3(k+1)}\right).
\eeq
On the other hand, since $N\geq \vp(3k+2)$, it follows that
\beq\label{lemma-meta-3}
\forall n\geq N\, \left(\|u_n-v_n\|\leq \frac{1}{3(k+1)}\right).
\eeq
Finally, apply \eqref{lemma-meta-1}, \eqref{lemma-meta-2} and \eqref{lemma-meta-3} to get that 
\[
\forall i,j\in[N, N+g(N)]\, \left(\|v_i-v_j\|\leq \frac{1}{k+1}\right).
\]
Since, obviously, $N\leq \Gamma(k,g)$, we conclude that $\Gamma$ is a rate of metastability 
for $(v_n)$.
\end{proof}

The first quantitative version of Theorem~\ref{main-theorem} is the following.

\begin{theorem}\label{main-theorem-quant-1}
Let $H$ be a Hilbert space, $A:H\to 2^H$ be a maximally monotone operator such that $zer(A)\ne\emptyset$, $(x_n)$ be defined 
by \eqref{HPPA-def-1} and  $b\in\N^*$  be such that $b\geq \max\{\|x_0-p\|, \ \|u-p\|\}$ for some zero $p$ of $A$. 
Suppose that $(C3_q)$, $(C4_q)$, $(C5_q)$ hold, $(\alpha_n)$ is nonincreasing and that $\ell\in\N, D\in\N^*$  satisfy \eqref{def-ell-D}.
Define 
\beq
\Omega(k,g) := \tilde{g}^{(9b^2(k+1)^2)}(0), \quad \text{~where~} \tilde{g}(n) := n+g(n). \label{def-Omega-main}
\eeq
The following hold:
\be
\item\label{main-theorem-quant-1-C1q} If $(C1_q)$ holds, then $(x_n)$ is a Cauchy sequence with rate of metastability  
$\Phi$ defined by 
\beq\label{def-Phi-C1q}
\Phi(k,g)=\max\left\{\Theta(3k+2),\ \Omega\left(3k+2,h_k\right)\right\},
\eeq
where $\Theta$ is given by \eqref{def-Theta-main} and   
\[
h_k(n) := \max\{\Theta(3k+2), n\}-n+g(\max\{\Theta(3k+2), n\}).
\]
\item\label{main-theorem-quant-1-C2q} If $(C2_q)$ holds and $\lsn_0:\N\to\N^*$ satisfies \eqref{hypothesis-lsn0},
then 
$(x_n)$ is Cauchy with rate of metastability  
$\widetilde{\Phi}$ defined by
\beq\label{def-Phi-C2q}
\widetilde{\Phi}(k,g)=\max\left\{\widetilde{\Theta}(3k+2),\ \Omega\left(3k+2,\widetilde{h}_k\right)\right\},
\eeq
where $\widetilde{\Theta}$ is given by \eqref{def-wTheta-main} and  
$$
\widetilde{h}_k(n) := \max\left\{\widetilde{\Theta}(3k+2), n\right\}-n+g\left(\max\left\{\widetilde{\Theta}(3k+2), n\right\}\right). 
$$
\ee
\end{theorem}
\begin{proof}
Since $(\alpha_n)$ is nonincreasing, we can apply Proposition~\ref{main-quant-zn}.\eqref{main-quant-zn-alphan-noninc} 
with $d:=3b$ to get that $(z_n)$ is Cauchy with rate of metastability $\Omega$.
Then, \eqref{main-theorem-quant-1-C1q}  is obtained by an  
application of Proposition~\ref{lemma-meta+conv-meta} and Proposition~\ref{lim-xn-zn-0-v1}.\eqref{lim-xn-zn-0-v1-C1q},
while \eqref{main-theorem-quant-1-C2q} follows from 
Proposition~\ref{lemma-meta+conv-meta} and Proposition~\ref{lim-xn-zn-0-v1}.\eqref{lim-xn-zn-0-v1-C2q}.
\end{proof}

The second quantitative version of Theorem~\ref{main-theorem} is obtained by considering $(C6_q)$ instead of $(C5_q)$.

\begin{theorem}\label{main-theorem-quant-2}
Let $H$ be a Hilbert space, $A:H\to 2^H$ be a maximally monotone operator such that $zer(A)\ne\emptyset$, $(x_n)$ be defined 
by \eqref{HPPA-def-1} and  $b\in\N^*$  be such that $b\geq \max\{\|x_0-p\|, \ \|u-p\|\}$ for some zero $p$ of $A$.
Suppose that $(C3_q)$, $(C4_q)$, $(C5_q)$ hold, $(\alpha_n)$ is nonincreasing and that $\ell\in\N, D^*\in\N^*$  satisfy 
\eqref{def-ell-Dstar}. Let $\Omega$ be given by \eqref{def-Omega-main}.

The following hold:
\be
\item\label{main-theorem-quant-2-C1q} If $(C1_q)$ holds, then $(x_n)$ is a Cauchy sequence with rate of metastability  
$\Phi^*$ defined by 
\beq\label{def-Phi-star-C1q}
\Phi^*(k,g)=\max\left\{\Theta^*(3k+2),\ \Omega\left(3k+2,h^*_k\right)\right\},
\eeq
where $\Theta^*$ is given by \eqref{def-Theta-star-main} and
$$
h^*_k(n) := \max\{\Theta^*(3k+2), n\}-n+g(\max\{\Theta^*(3k+2), n\}).
$$
\item\label{main-theorem-quant-2-C2q} If $(C2_q)$ holds and $\lsn^*_0$ satisfies \eqref{hypothesis-lsnstar0}, then 
$(x_n)$ is Cauchy with rate of metastability  $\widetilde{\Phi}$  defined by 
\beq\label{def-Phi-star-C2q}
\widetilde{\Phi^*}(k,g)=\max\left\{\widetilde{\Theta^*}(3k+2),\ \Omega\left(3k+2,\widetilde{h^*_k}\right)\right\},
\eeq
where $\widetilde{\Theta^*}$ is given by \eqref {def-wTheta-star-main} and 
$$
\widetilde{h^*_k}(n) := \max\left\{\widetilde{\Theta^*}(3k+2), n\right\}-n+
g\left(\max\left\{\widetilde{\Theta^*}(3k+2), n\right\}\right).
$$
\ee
\end{theorem}
\begin{proof}
Proceed as in the proof of Theorem~\ref{main-theorem-quant-1}. 
The only difference is that we use  Proposition~\ref{lim-xn-zn-0-v2} instead of Proposition~\ref{lim-xn-zn-0-v1}.
\end{proof}

As one can see from the proofs of the previous results, the hypothesis that $(\alpha_n)$ is nonincreasing, 
appearing in both our main theorems, is used only for computing a rate of metastability for $(z_n)$. 
However, the statement of Proposition~\ref{main-quant-zn} shows that such a rate of metastability 
can be also computed if we use $(C0^*_{q})$. As a consequence, we can obtain slightly 
changed versions of Theorems~\ref{main-theorem-quant-1}, \ref{main-theorem-quant-2} using 
 $(C0^*_{q})$ as a hypothesis instead of $(\alpha_n)$ being nonincreasing.

Let us give some consequences of Theorem~\ref{main-theorem-quant-1}. One can obtain in a similar way 
corollaries of Theorem~\ref{main-theorem-quant-2}.

By letting $e_n=0$, we get that 
\beq
x_n=\alpha_n u+(1-\alpha_n)J_{\beta_n A}x_n. \label{def-HPPA-en0}
\eeq
Furthermore, both $(C5_q)$ and $(C6_q)$ hold trivially with $\ratecfiveq=\ratecsixq=0$.

\bcor\label{meta-theorem-2-c1}
Let $H, A, b$ be as in Theorem~\ref{main-theorem-quant-1} with $(x_n)$ defined by \eqref{def-HPPA-en0}.
Suppose that $(C3_q)$, $(C4_q)$ hold, $(\alpha_n)$ is nonincreasing and that $\ell\in \N^*$ is such 
that $\beta\geq \frac{1}{\ell+1}$.

The following hold:
\be
\item If $(C1_q)$ holds, then $(x_n)$ is a Cauchy sequence with rate of metastability  
$\Phi$ defined by \eqref{def-Phi-C1q}, with $\Omega$, $h_k$, $\psi$ as in Theorem~\ref{main-theorem-quant-1} and 
\bea
\Theta(k) &:=& \rateconeq(\lsn(k)+\lceil \ln(3(1+5b)(k+1))\rceil)+1, \label{def-Theta-cor1}\\
\lsn(k) &:=& \max\{\psi(3k+2), 1\}.
\eea
\item If $(C2_q)$ holds and $\lsn_0:\N\to\N^*$ satisfies \eqref{hypothesis-lsn0},
then  $(x_n)$ is Cauchy with rate of metastability  
$\widetilde{\Phi}$ defined by \eqref{def-Phi-C2q}, with $\Omega$, $\widetilde{h}_k$, $\psi$ as 
in Theorem~\ref{main-theorem-quant-1}, 
$\lsn$ as in (i) and 
\beq \label{def-wTheta-cor1}
\widetilde{\Theta}(k) := \max\left\{\ratectwoq\left(3(1+5b)\lsn_0(k)(k+1)-1\right),\lsn(k)\right\}+1.
\eeq
\ee
\ecor
\begin{proof}
Apply Theorem~\ref{main-theorem-quant-1} with $D=1$ and $\ratecfiveq=0$.
\end{proof}

By letting $g(n)=L$, we get uniform rates of $L$-metastability for every $L\in\N$.

\bcor\label{meta-theorem-2-c2}
Assume the hypothesis of Corollary~\ref{meta-theorem-2-c1} and let $L\in\N$. Define 
\bua
\Delta_L:\N\to \N, \quad \Delta_L(k):=\Theta(3k+2)+R(k)L\\
\widetilde{\Delta}_L:\N\to \N, \quad \widetilde{\Delta}_L(k):=\widetilde{\Theta}(3k+2)+R(k)L,
\eua
 where $\Theta$ is given by \eqref{def-Theta-cor1}, $\widetilde{\Theta}$ is given by \eqref{def-wTheta-cor1} and 
 $R:\N\to \N$ is defined by $R(k):=81b^2(k+1)^2$. 
 
The following hold:
\be
\item\label{meta-theorem-2-c2-1} If $(C1_q)$ holds, then $\Delta_L$ is a rate of $L$-metastability of $(x_n)$.
\item\label{meta-theorem-2-c2-2}  If $(C2_q)$ holds, then $\widetilde{\Delta}_L$ is a rate of $L$-metastability of $(x_n)$.
\ee
\ecor
\begin{proof}
We apply Corollary~\ref{meta-theorem-2-c1} with $g(n):=L$ for all $n\in\N$. 
\be
\item Remark, using the notations from Theorem~\ref{main-theorem-quant-1}.\eqref{main-theorem-quant-1-C1q}, that
\bua 
h_k(n) & =& \max\{\Theta(3k+2),n\}-n + L , \\
 \tilde{h_k}(n) & =& \max\{\Theta(3k+2),n\} + L.
\eua 
One can easily see by induction that $\tilde{h_k}^{(n)}(0)=\Theta(3k+2)+nL$ for all $n\geq 1$. It follows that 
\bua 
\Omega\left(3k+2,h_k \right)& = &  \tilde{h_k}^{(R(k))}(0) = \Theta(3k+2)+R(k)L, \\
\Phi(k,g) & = & \max\left\{\Theta(3k+2),\ \Omega\left(3k+2,h_k\right)\right\} =  \Theta(3k+2)+R(k)L \\
& = &  \Delta_L(k).
\eua
\item  The proof is similar. 
\ee
\end{proof}

Recently, Kohlenbach \cite{Koh20a} computed rates of metastability for the particular version of 
the HPPA $(x_n)$ obtained by letting $e_n:=0$ for all $n\in\N$ (as in \eqref{def-HPPA-en0}), 
in the more general setting of accretive operators in uniformly convex and uniformly smooth Banach spaces.  
The strong convergence proof analyzed in \cite{Koh20a}, due to Aoyama and Toyoda \cite{AoyToy17}, 
uses conditions $(C0)$, $(C1)$ on $(\alpha_n)$ and the hypothesis $\inf_n\beta_n>0$ on $(\beta_n)$ and is different from the one we analyze in this paper.
In order to compute the rate of metastability for $(x_n)$, Kohlenbach uses the quantitative form $(C0_q)$,  requiring a rate of convergence of $(\alpha_n)$. We compute rates of metastability for $(x_n)$ by using, 
instead of $(C0_q)$, either the hypothesis that $(\alpha_n)$ is nonincreasing  or 
$(C0_q^*)$, a weaker form of $(C0_q)$ which needs only a quasi-rate of convergence of $(\alpha_n)$. 
Furthermore, the rate of metastability computed in \cite{Koh20a}  depends on an extra-sequence 
$(\widetilde{\alpha}_n)$ satisfying $0<\widetilde{\alpha}_n\leq \alpha_n$ for all $n\in\N$. 
As  a consequence, as Kohlenbach also points out, the bounds obtained by him for the sequence 
$(x_n)$, given by  \eqref{def-HPPA-en0}, are more complicated than  the ones we obtain in this paper.

\section{An example}\label{hppa-example}

We finish with an example of parameters $(\alpha_n), (\beta_n), (e_n)$ satisfying the hypotheses 
of Theorems~\ref{main-theorem-quant-1} and \ref{main-theorem-quant-2}.

Let 
\[
\alpha_n:=(n+2)^{-3/4}, \quad \beta_n:=1+\frac{(-1)^n}{n+1}, \quad e_n:=0  \quad \text{~for all~} n\in\N.
\]
One can verify that 
\be
\item $(\alpha_n)$ is nonincreasing.
\item $(C0_q)$ holds with $\rateczeroq(k)=(k+1)^2$. 
\item $(C1_q)$ holds with $\rateconeq(k)=(k+1)^4$.
\item $(C2_q)$ holds with $\ratectwoq(k)=k$.
\item $(C3_q)$ holds with $\ratecthreeq(k)= (k+1)^4+1$.
\item $(C4_q)$ holds with $\beta=1$ and $\ratecfourq(k)=k$.
\item $(C5_q)$ and $(C6_q)$ hold with $\ratecfiveq(k)=\ratecsixq(k)=0$.
\ee
Furthermore, $\ell=0$ and $D=1$ satisfy \eqref{def-ell-D} and, with the notations of 
Proposition~\ref{lim-xn-zn-0-v1}.\eqref{lim-xn-zn-0-v1-C1q},
\bua
\psi(k) &=& (6b(k+1))^4+1=6^4b^4(k+1)^4+1, \\
\lsn(k) &=& 18^4b^4(k+1)^4+1,\\
\Theta(k) &=& \left(18^4b^4(k+1)^4+1+\lceil \ln(3(1+5b)(k+1))\rceil\right)^4+1.
\eua

\bprop\label{example-xn-zn}
$\limn \|x_n-z_n\|=0$ with rate of convergence 
$$
\Theta_0(k) = \left(18^4b^4(k+1)^4+18b(k+1)+1\right)^4+1.
$$
\eprop
\begin{proof} By Proposition~\ref{lim-xn-zn-0-v1}.\eqref{lim-xn-zn-0-v1-C1q} and the fact that $\Theta(k)\leq \Theta_0(k)$. 
\end{proof}

Applying Theorem~\ref{main-asympt-reg} with $\Lambda=\Theta_0$ and $M_i=1$ for every $i\in\N$, 
we get effective rates of asymptotic regularity.

\bprop\label{example-rate-as-reg}
Let $C:=72b$ and define $\exasreg, \, \exasreg^*:\N\to\N$ by 
\bua
\exasreg(k) &:=&  \left(C^4(k+1)^4+C(k+1)+1\right)^4+1, \\
\exasreg^*(k) &:=& \exasreg(2k+1) =\left(16C^4(k+1)^4 + 2C(k+1)+1\right)^4+1.
\eua
Then 
\be
\item $\exasreg$  is a rate of asymptotic regularity of $(x_n)$ with respect to $J_{\beta A}$.
\item $\exasreg^*$  is a rate of asymptotic regularity of $(x_n)$ with respect to $(J_{\beta_n A})$ and 
$J_{\beta_i A}$ ($i\in\N$).
\ee
\eprop

Finally, an application of Corollary~\ref{meta-theorem-2-c2}.\eqref{meta-theorem-2-c2-1}
gives us the following result.

\bprop\label{example-stable-constant-g}
Let $L\in\N$ and define $\exLmeta:\N\to\N$ by 
\[ \exLmeta(k):=\left(54^4b^4(k+1)^4+54b(k+1)+1\right)^4 + 81b^2(k+1)^2L+1.\]
Then $\exLmeta$ is a rate of $L$-metastability for $(x_n)$, that is: for all $k\in\N$, there exists 
$N\leq \exLmeta(k)$ such that 
\[\|x_i-x_j\|\leq \frac{1}{k+1} \qquad \text{for all~} i,j\in [N,N+L].\]
\eprop
\begin{proof}
Let $k\in\N$ be arbitrary. By Corollary~\ref{meta-theorem-2-c2}.\eqref{meta-theorem-2-c2-1}, we get that there exists 
$N\leq \Theta(3k+2)+R(k)L$ such that 
\[\|x_i-x_j\|\leq \frac{1}{k+1} \qquad \text{for all~} i,j\in [N,N+L].\]
Remark that $\Theta(3k+2)+R(k)L\leq \Theta_0(3k+2)+R(k)L= \exLmeta(k)$.
\end{proof}

Thus, we obtain for this example, as a consequence of the quantitative results from Section~\ref{section-main},
polynomial rates of convergence for $(\|x_n-z_n\|)$, polynomial rates of asymptotic regularity 
for $(x_n)$ and polynomial rates of $L$-metastability of $(x_n)$ for every $L\in\N$. 
Furthermore, the rates are highly uniform: they depend on the Hilbert space $H$, the maximally 
monotone operator $A$ and the sequence $(x_n)$ only via $b$, an upper bound on $\|x_0-p\|, \|u-p\|$ 
for some zero $p$ of $A$.

\section*{Acknowledgements}

Lauren\c{t}iu Leu\c{s}tean was partially supported by a grant of the Romanian Ministry of Research and Innovation, Program 1 - Development of the National RDI System, Subprogram 1.2 - Institutional Performance -  Projects for Funding the Excellence in RDI, contract number 15PFE/2018. Pedro Pinto acknowledges and is thankful for the financial support of: FCT - Funda\c{c}\~{a}o para a Ci\^{e}ncia e Tecnologia under the project UID/MAT/04561/2019; the research center Centro de Matem\'{a}tica, Aplica\c{c}\~{o}es Fundamentais com Investiga\c{c}\~{a}o Operacional, Universidade de Lisboa; and the `Future Talents' short-term scholarship at Technische Universit\"at Darmstadt.
The paper also benefited from discussions with Fernando Ferreira and Ulrich Kohlenbach.

\end{document}